\documentclass[fleqn,reqno,10pt,a4paper,final]{amsart}

\usepackage[a4paper,left=35mm,right=35mm,top=30mm,bottom=30mm,marginpar=25mm]{geometry} 
\usepackage{amsmath}
\usepackage{amssymb}
\usepackage{amsthm}
\usepackage{amscd}
\usepackage[ansinew]{inputenc}
\usepackage{cite}
\usepackage{bbm}
\usepackage{color}
\usepackage[english=american]{csquotes}
\usepackage[final]{graphicx}
\usepackage{hyperref}
\usepackage{soul}

\graphicspath{{Pictures/}}

\numberwithin{equation}{section}

\newtheoremstyle{thmlemcorr}{10pt}{10pt}{\itshape}{}{\bfseries}{.}{10pt}{{\thmname{#1}\thmnumber{ #2}\thmnote{ (#3)}}}
\newtheoremstyle{thmlemcorr*}{10pt}{10pt}{\itshape}{}{\bfseries}{.}\newline{{\thmname{#1}\thmnumber{ #2}\thmnote{ (#3)}}}
\newtheoremstyle{defi}{10pt}{10pt}{\itshape}{}{\bfseries}{.}{10pt}{{\thmname{#1}\thmnumber{ #2}\thmnote{ (#3)}}}
\newtheoremstyle{remexample}{10pt}{10pt}{}{}{\bfseries}{.}{10pt}{{\thmname{#1}\thmnumber{ #2}\thmnote{ (#3)}}}
\newtheoremstyle{ass}{10pt}{10pt}{}{}{\bfseries}{.}{10pt}{{\thmname{#1}\thmnumber{ A#2}\thmnote{ (#3)}}}

\theoremstyle{thmlemcorr}
\newtheorem{theorem}{Theorem}
\numberwithin{theorem}{section}
\newtheorem{lemma}[theorem]{Lemma}
\newtheorem{corollary}[theorem]{Corollary}
\newtheorem{proposition}[theorem]{Proposition}

\newtheorem{question}[theorem]{Question}

\theoremstyle{thmlemcorr*}
\newtheorem{theorem*}{Theorem}
\newtheorem{lemma*}[theorem]{Lemma}
\newtheorem{corollary*}[theorem]{Corollary}
\newtheorem{proposition*}[theorem]{Proposition}
\newtheorem{problem*}[theorem]{Problem}
\newtheorem{conjecture*}[theorem]{Conjecture}

\theoremstyle{defi}
\newtheorem{definition}[theorem]{Definition}

\theoremstyle{remexample}
\newtheorem{remark}[theorem]{Remark}

\theoremstyle{ass}
\newtheorem{assumption}{Assumption}

\newcommand{\Ecal}{\mathcal{E}}

\newcommand{\Hcal}{\mathcal{H}}

\newcommand{\Vcal}{\mathcal{V}}

\newcommand{\Xcal}{\mathcal{X}}
\newcommand{\Ycal}{\mathcal{Y}}

\newcommand{\Sbb}{\mathbb{S}}

\DeclareMathOperator{\Var}{Var}

\newcommand{\norm}[1]{\|#1\|}

\newcommand{\abs}[1]{|#1|}

\newcommand{\di}{\mathrm{d}}

\newcommand{\N}{\mathbb{N}}
\newcommand{\R}{\mathbb{R}}

\def\XXint#1#2#3{{\setbox0=\hbox{$#1{#2#3}{\int}$} 
\vcenter{\hbox{$#2#3$}}\kern-.5\wd0}}

\newcommand{\p}{\partial}

\renewcommand{\phi}{\varphi}

\newcommand{\pa}[1]{\left( #1 \right)}
\newcommand{\rpa}[1]{\left[ #1 \right]}
\newcommand{\br}[1]{\left\lbrace #1\right\rbrace}

\title{Weak Poincar\'e  inequalities  in the absence of spectral gaps}

\author{Jonathan Ben-Artzi}

\address{School of  Mathematics, Cardiff University, Cardiff CF24 4AG, Wales, United Kingdom.}
\email{Ben-ArtziJ@cardiff.ac.uk}

\author{Amit Einav}

\address{Institut f\"ur Analysis und Scientific Computing,
Technische Universit\"at Wien, Austria.}
\email{amit.einav@asc.tuwien.ac.at}

\begin{document}
\maketitle


\begin{abstract}
For generators of Markov semigroups which lack a spectral gap, it is shown how bounds on the density of states  near zero lead to  a so-called ``weak Poincar\'e inequality'' (WPI), originally introduced by Liggett [Ann. Probab., 1991]. Applications to {general classes of constant coefficient pseudodifferential operators are studied. Particular examples are }the heat semigroup and   the semigroup generated by the fractional Laplacian {in the whole space}, where the optimal decay rates are recovered. {Moreover, }the classical Nash inequality appears as a special case of the WPI for the heat semigroup.

\vspace{4pt}



\vspace{4pt}

\noindent\textsc{Date:} \today.

\vspace{4pt}

\noindent\textsc{Acknowledgements:} The first author was support by the UK Engineering and Physical Sciences Research Council (EPSRC) grant EP/N020154/1. The second author was supported by the Austrian Science Fund (FWF) grant M 2104-N32. {The authors thank the referees for their useful comments which helped improve the content and presentation of this work.}
\end{abstract}

\setcounter{tocdepth}{1} 
\tableofcontents

\section{Introduction and statement of results}\label{sec:intro}
In this note we study how the well-known equivalence between spectral gaps, Poincar\'e inequalities and exponential rates of decay to equilibrium extends to systems which lack a spectral gap but have a bounded density of states near $0$. Our main result relies solely on our ability to ``differentiate'' the resolution of the identity of a given operator. It is thus  quite general, and covers important examples such as Markov semigroups.

Our setup is as follows: we let $M$ be a manifold with Borel measure $\di\mu$, $\Hcal=L^2(M,\di\mu;\R)$ equipped with scalar product $(\cdot,\cdot)_\Hcal$. We assume that $H:D(H)\subset\Hcal\to\Hcal$ is a self-adjoint, non-negative operator, so that $-H$ is the infinitesimal generator of a Markov semigroup $(P_t)_{t\geq0}$, whose invariant measure is $\di\mu$, i.e. for every $u$ that is bounded and non-negative  $\int_M P_tu\,\di\mu=\int_M u\,\di\mu$ for any $t\geq0.$
 Let $\{E(\lambda)\}_{\lambda\geq0}$ be the resolution of the identity of $H$ and let the associated Dirichlet form be 
\[\Ecal(u):=\int_{M}(H^{1/2}u)^2\,\di\mu.\]

As stated above, instead of assuming a spectral gap, we assume the opposite: $H$ has continuous spectrum in a neighborhood of $0$ (and $0$ itself is possibly an eigenvalue). We show that an appropriate estimate of the density of the spectrum near $0$ leads to a weaker version of the Poincar\'e inequality (also known as a weak Poincar\'e inequality, defined below in Definition  \ref{def:wpi}). This, in turn, leads to an algebraic decay rate for the associated semigroup.

In this paper we employ the following definition for the variance  of a given function $u\in\Hcal$:
 \begin{equation*}
 \Var(u):=\int_M(u-E(\{0\})u)^2\,\di\mu
 \end{equation*}
 where $E(\{0\})$ is the projection onto the kernel of $H$.  In the case where the kernel only consists of constant functions and $\mu$ is a probability measure, this definition coincides with the standard definition, see \cite[\S4.2.1]{Bakry2014}. We discuss the significance of the resolution of the identity of $H$ (and in particular the projection onto its kernel) and its relationship with functional inequalities and decay rates below  in Section \ref{sec:weak}.
 
%
%
%
%
%
%
%
%
%

 We can now  recall  the classical Poincar\'e inequality (again, see  \cite[\S4.2.1]{Bakry2014}):
 \begin{definition}[Poincar\'e Inequality]\label{def:poincare}
We say that  $H$ satisfies a \emph{Poincar\'e inequality} if there exists $C>0$ such that
\begin{equation*}\label{eq:poincare_general_H}
\Var(u) \leq C \Ecal(u),\qquad \forall u\in D(\Ecal),
\end{equation*}
where $C$ does not depend on $u$.
\end{definition}

\begin{remark}
The topology of $D(\Ecal)$ is the graph norm  topology generated by $\|\cdot\|_\Hcal^2+\Ecal(\cdot)$, see \cite[\S3.1.4]{Bakry2014}.
\end{remark}

The definition of  a ``weak Poincar\'e inequality'' is somewhat ambiguous. This is  addressed  in further detail  in Section \ref{sec:weak} below. We adopt the following  definition, motivated by Liggett \cite[Equation (2.3)]{Liggett1991}:

\begin{definition}[Weak Poincar\'e Inequality]\label{def:wpi}
Let $\Phi:\Hcal\to[0,\infty]$ satisfy $\Phi(u)<\infty$ on a dense subset of $D(\Ecal)$. Let $p\in(1,\infty)$. We say that  $H$ satisfies a \emph{$(\Phi,{p})$-weak Poincar\'e inequality ($(\Phi,{p})$-WPI)}  if there exists $C>0$ such that
\begin{equation}\label{eq:wpi}
\Var(u)\leq C\Ecal(u)^{1/p} \Phi(u)^{1/q},\qquad\forall u\in D(\Ecal),
\end{equation}
where $C$ does not depend on $u$ and where $1/p+1/q=1$.
\end{definition}

\begin{remark}
Note that \eqref{eq:wpi} is meaningful only  on a dense subset of $D(\Ecal)$ where $\Phi<+\infty$.
\end{remark}

\subsection{The Hilbertian case}
{We start our discussion by considering the purely Hilbertian case, i.e. we consider generators with density of states that are defined on} subspaces which respect the Hilbert structure of $\Hcal$, such as Sobolev spaces or weighted spaces. Our basic assumption is:
\begin{assumption}\label{ass:1}
There exists a dense subspace $\Xcal\subset\Hcal$ such that
\begin{enumerate}
\item $\Xcal\cap D(\Ecal)$ is dense in $D(\Ecal)$  (in the topology of $D(\Ecal)$),
\item for some constants $r>0$, $C_1>0$ and $\alpha>-1$,
\end{enumerate}
 the mapping  $\lambda\mapsto\frac{\di}{\di\lambda}(E(\lambda)u,v)_{\Hcal}$ is continuous on $(0,r)$ for every $u,v\in\Xcal$ and satisfies
	\begin{equation}\label{eq:dens-states}
	\left|\frac{\di}{\di\lambda}(E(\lambda)u,v)_{\Hcal}\right|\leq C_1\lambda^\alpha\|u\|_\Xcal\|v\|_\Xcal,\qquad\forall u,v\in\Xcal,\,\forall \lambda\in(0,r).
	\end{equation}	
\end{assumption}
\begin{remark}\label{rem:density_of_states_notion}
We refer to the bilinear form  $\frac{\di}{\di\lambda}(E(\lambda)\cdot,\cdot)_{\Hcal}$ as the density of states (DoS) of $H$ at $\lambda$. Note that  if the DoS satisfies a bound as in \eqref{eq:dens-states} and $\Xcal$ has a norm compatible with (and stronger than) the norm on $\Hcal$ then  it induces an operator $\Xcal\to\Xcal^*$ by the Riesz representation theorem.
\end{remark}

We can finally state our main results on how \eqref{eq:dens-states} leads to a $(\Phi,p)$-WPI (Theorem \ref{thm:main}) and, in turn, an explicit rate of decay (Theorem \ref{thm:main1}). Theorem \ref{thm:main} will be further generalized below in Theorem \ref{thm:main2}{, and then again in Proposition \ref{prop:main2_improved} where  a precise constant in the WPI is obtained}. The decay rates presented in Theorem \ref{thm:main1} apply to the Markov semigroup generated by $H$.
\begin{theorem}\label{thm:main}
If Assumption A\ref{ass:1} holds  then $H$ satisfies a $(\Phi,p)$-weak Poincar\'e inequality with $\Phi(u)=\|u\|_{\Xcal}^2$ (and $\Phi(u)=+\infty$ if $u\in\Hcal\setminus\Xcal$) and $p=\frac{2+\alpha}{1+\alpha}$. 
\end{theorem}

\begin{theorem}\label{thm:main1}
Let Assumption A\ref{ass:1} hold. Let $u\in\Xcal$ and suppose that there exist $C_2=C_2(u)\geq0$ and $\beta\in \R$, such that the Markov semigroup satisfies
	\begin{equation}\label{eq:semigroup-decay-ass}
	\norm{P_tu}^2_{\Xcal}\leq \norm{u}^2_{\Xcal}+C_2t^\beta,\qquad \forall t\geq0.
	\end{equation}
Then
	\begin{equation}\label{eq:bound-variance}
	\Var(P_t u)
	\leq
	\left(\Var(u) ^{\frac{-1}{1+\alpha}}+C_3
	\int_{0}^t {({\norm{u}^2_{\Xcal}+C_2s^\beta})^{\frac{-1}{1+	\alpha}}{\,\di s}}\right)^{-({1+\alpha})}
	\end{equation}
where $C_3$ is given explicitly (and only depends on $\alpha,\, C_1$). In particular, $\Var(P_t u)$ satisfies the following decay rates as $t\to+\infty$:
	\[
	\Var({P_t u})
	\leq
	\begin{cases} 
	O({(\log t) ^{-(1+\alpha)}}) & \beta=1+\alpha.\\
	O({t^{\beta-(1+\alpha)}}) & 0<\beta <1+\alpha.\\
	O({t^{-(1+\alpha)}}) & C_2=0 \text{ or }\beta\leq 0.
	\end{cases}
	\]
\end{theorem}

 \begin{remark}
1. The choice of space $\Xcal$ is motivated by \eqref{eq:semigroup-decay-ass}: it is beneficial to choose $\Xcal$ that is invariant under the Markov semigroup {(i.e., if $u\in\Xcal$ then $P_tu\in\Xcal$ for all $t\geq0$)}.

2. Clearly $C_2(u)$ is subject to quadratic scaling, for example it can be $C\|u\|_\Hcal^2$ or $C\|u\|_\Xcal^2$, but the explicit form is not important.
 \end{remark}

\subsection{A generalized theorem: departing from the Hilbert structure}
Theorems \ref{thm:main} and \ref{thm:main1} demonstrate  how estimates on the density of states near $0$ imply a weak Poincar\'e inequality and a rate of decay to equilibrium. However it is not essential to restrict oneself to a  subspace $\Xcal$. In fact, it is often desirable to deal with functional spaces that are not contained in $\Hcal$, as it may provide improved estimates and decay rates. In particular, this makes sense when the operator in question is the generator of a Markov semigroup, and acts on a range of spaces simultaneously. Hence we replace Assumption A\ref{ass:1} by a more general one:
\begin{assumption}\label{ass:2}
There exist  Banach spaces $\Xcal,\Ycal$ of functions on $M$, a  constant $r>0$ and a function $\psi_{\Xcal,\Ycal}\in L^1(0,r)$ that is strictly positive a.e. on $(0,r)$,  such that

\begin{enumerate}
\item $\Xcal\cap\Ycal\cap D(\Ecal)$ is dense in $D(\Ecal)$ (in the topology of $D(\Ecal)$).

\item The mapping  $\lambda\mapsto\frac{\di}{\di\lambda}(E(\lambda)u,v)_{\Hcal}$ is continuous on $(0,r)$ for every $u\in\Xcal\cap\Hcal$ and $v\in\Ycal\cap\Hcal$ and satisfies
	\begin{equation}\label{eq:dens-states2}
	\left|\frac{\di}{\di\lambda}(E(\lambda)u,v)_{\Hcal}\right|\leq  \psi_{\Xcal,\Ycal}(\lambda)\|u\|_\Xcal\|v\|_\Ycal,\qquad\forall \lambda\in(0,r).
	\end{equation}	
\end{enumerate}
\end{assumption}

We can now state the following more general theorem:

\begin{theorem}\label{thm:main2}
Let the conditions of  Assumption A\ref{ass:2} hold, and define $\Psi_{\Xcal,\Ycal}(\rho)=\int_0^\rho\psi_{\Xcal,\Ycal}(\lambda)\,\di\lambda$, $\rho\in(0,r)$. Then:

a.  There exists $K_0\in(0,1)$ such that the following functional inequality holds:
	\begin{equation}\label{eq:wpi-implicit}
	(1-K)\Psi_{\Xcal,\Ycal}^{-1}\left(K\frac{\Var(u)}{\|u\|_\Xcal\|u\|_\Ycal}\right)\Var(u)\leq\Ecal(u),\qquad\forall K\in(0,K_0),\,\forall u\in D(\Ecal)
	\end{equation}
where $\|u\|_\Xcal=+\infty$ if $u\notin \Xcal$ and similarly for $\Ycal$.

b. If $\Xcal=\Ycal$ and $\psi_{\Xcal,\Ycal}(\lambda)=C_1\lambda^\alpha$, $\alpha>-1$, the estimate \eqref{eq:wpi-implicit} reduces to the $(\Phi,p)$-WPI as in Definition \ref{def:wpi} with $\Phi(u)=\|u\|_\Xcal^2$ and $p=\frac{\alpha+2}{\alpha+1}$.

c. If, in addition, $\Xcal=\Ycal\subset\Hcal$ then we obtain Theorem \ref{thm:main}.
\end{theorem}

\begin{remark}
The inequality \eqref{eq:wpi-implicit} can be viewed as an implicit form of the weak Poincar\'e inequality. Note that setting $K=0$ (which is excluded in the theorem) leads to the Poincar\'e inequality.
\end{remark}

The power of this result is demonstrated in the following corollary, where the celebrated Nash inequality is obtained as a simple consequence. This simple derivation is discussed  in Remark \ref{rek:nash} below.
\begin{corollary}[Nash inequality]\label{cor:nash}
When $H=-\Delta:H^2(\R^d)\subset L^2(\R^d)\to L^2(\R^d)$ and $\Ycal=\Xcal=L^1(\R^d)$ the  inequality \eqref{eq:wpi-implicit} is precisely Nash's inequality  \cite{Nash1958}:
	\begin{equation*}\label{eq:nash}
	\|u\|_{L^2}^2
	\leq
	C
	\left(\|\nabla u\|_{L^2}^2\right)^{\frac{d}{d+2}}
	\left(\|u\|_{L^1}^2\right)^{\frac{2}{d+2}},\qquad\forall u\in L^1(\R^d)\cap H^1(\R^d),
	\end{equation*}
where $C>0$ does not depend on $u$.  {Furthermore, using Proposition \ref{prop:main2_improved} (below)  an explicit constant may be computed to yield $C=\pa{\frac{\abs{\Sbb^{d-1}}}{2}}^{\frac{2}{2+d}}\frac{2+d}{d}$.}
\end{corollary}

\begin{proof}
The (simple) proof of this corollary is {done by applying our results to the heat semigroup. More details are }provided in the examples (Section \ref{sec:examples}), in particular see Remark \ref{rek:nash}.
\end{proof}

\begin{remark}\label{rem:genera_psi}
The requirement that $\psi_{\Xcal,\Ycal}$ is strictly positive a.e. on $(0,r)$, for some $r>0$ (perhaps very small), is quite natural as we are interested in operators that lack a spectral gap. However, one can easily generalize our result even if that is not the case by defining
\[\Psi_{\Xcal,\Ycal}^{-1}(y)=\sup\br{x\in (0,r)  \;|\; \Psi_{\Xcal,\Ycal}(x) \leq y}.\]
\end{remark}

\subsection{Precise constants}
{Under additional mild assumptions one can improve Theorem \ref{thm:main2} by replacing the inequality \eqref{eq:wpi-implicit}  which contains an arbitrary constant $K$ with an inequality that has an explicit constant.} {The question of how far this constant is from being sharp is the topic of ongoing research.}
\begin{proposition}\label{prop:main2_improved}
{Let the conditions of Assumption A\ref{ass:2} hold. Assume in addition that $\psi_{\Xcal,\Ycal}$ can be extended to a continuous function on $(0,R)$, where $R\in[r,+\infty]$ is such that if $\Psi_{\Xcal,\Ycal}(\rho):=\int_0^\rho\psi_{\Xcal,\Ycal}(\lambda)\,\di\lambda$, $\rho\in(0,R)$ and
\[g_{\Xcal,\Ycal}(\rho):=\Psi_{\Xcal,\Ycal}(\rho)+\rho\psi_{\Xcal,\Ycal}(\rho)\]
then $g$  is non-decreasing and $\lim_{\rho\rightarrow 0^{+}}g_{\Xcal,\Ycal}(\rho)=0$, $\lim_{\rho\rightarrow R^{-}}g_{\Xcal,\Ycal}(\rho)=+\infty$.
Then:}

{a. The following functional inequality holds:
	\begin{equation}\label{eq:wpi-improved}
	 \pa{g_{\Xcal,\Ycal}^{-1}\pa{\frac{\Var(u)}{\|u\|_\Xcal\|u\|_\Ycal}}}^2
	\psi_{\Xcal,\Ycal}\pa{g_{\Xcal,\Ycal}^{-1}\pa{\frac{\Var(u)}{\|u\|_\Xcal\|u\|_\Ycal}}}\|u\|_\Xcal\|u\|_\Ycal
	\leq\mathcal{E}(u),\quad\forall u\in D(\Ecal),
\end{equation}
where $\|u\|_\Xcal=+\infty$ if $u\notin \Xcal$ and similarly for $\Ycal$.}

{b. If $\Xcal=\Ycal$, and $\psi_{\Xcal,\Ycal}(\lambda)=C_1\lambda^\alpha$, $\alpha>-1$ then the estimate \eqref{eq:wpi-improved} reduces to the $(\Phi,p)$-WPI as in Definition \ref{def:wpi} with $\Phi(u)=\|u\|_\Xcal^2$, $p=\frac{\alpha+2}{\alpha+1}$ and $C=C_1^{\frac{1}{2+\alpha}}\frac{2+\alpha}{1+\alpha}$.}
\end{proposition}


\noindent\textbf{Organization of the paper.} Before proceeding to prove our theorems we first discuss both the classical and the weak Poincar\'e inequalities, and their connection to Markov semigroups in Section \ref{sec:poincare}. The proofs will follow in Section \ref{sec:proof} and we  then present various applications of these theorems in Section \ref{sec:examples}, where we shall also prove Corollary \ref{cor:nash}.

\section{Poincar\'e inequalities}\label{sec:poincare}
In this section we recall the famous Poincar\'e inequality, its connection to Markov semigroups, and we  discuss its ``weak'' variant, the so-called ``weak Poincar\'e inequality.'' \subsection{The classical Poincar\'e inequality}\label{sec:classical-poincare}
When $M$ is a compact Riemannian manifold or a bounded domain of $\R^d$, the classical  $L^2$ Poincar\'e inequality  reads  \cite[\S4.2.1]{Bakry2014}
	\begin{equation}\label{eq:poincare_original}
	\int_{M}\left|\phi(x)-\pa{\frac{1}{\abs{M}}\int_{M}\phi(y)\,\di y}\right|^2 \,\di x \leq C_{M}\int_{M}\abs{\nabla \phi(x)}^2\, \di x,
	\end{equation}
where $\abs{M}$ is the volume of  $M$, and $C_{M}>0$  is independent of $u$.\\

\emph{Motivation: the heat semigroup.} Let us illustrate why the quantities appearing in this inequality are   natural. Let $M\subset\R^d$ be a bounded, connected and smooth domain. Consider the heat semigroup, i.e. solutions  of 
	\[
	\partial_t u(t,x)=\Delta_x u(t,x),\qquad x\in M,\,t\in\R_+,
	\]
subject to Neumann boundary conditions with initial data $u(0,x)=u_0(x)$. The associated invariant measure is $\di\mu(x) = \frac{\di x}{\abs{M}}.$
It is well-known that in this case the spectrum of $\Delta_x$  is discrete and non-positive. In particular, its kernel is separated from the rest of the spectrum. This immediately implies that $u(t,x)=P_t u_0(x)$ converges to the projection onto the kernel, given by
	\[
	P_{\ker} u_0 : =  \int_{M}u_0(x)\,\di \mu(x).
	\]
Thus, we are interested in the  decay rate as $t\to+\infty$ of 
	\begin{equation*}
	\Vcal(P_tu_0):=\norm{P_t u_0-P_{\ker}\pa{P_t u_0}}^2_{L^2\pa{\di\mu}}=\norm{P_t u_0-P_{\ker}{u_0}}^2_{L^2\pa{\di\mu}}.\\
	\end{equation*}

\emph{The entropy method.} A common method to obtain decay rates of this type  is the so-called \emph{entropy method}. Given the  ``relative distance" $\Vcal$ (a Lyapunov functional) we find its \emph{production functional} $\Ecal$ by formally differentiating along the flow of the semigroup:
	\begin{equation}\label{eq:diff_ineq}
	\frac{\di}{\di t} \Vcal(P_tu_0) = 2 \left(\p_tP_t u_0,P_t u_0-P_{\ker}u_0\right)_{L^2(\di\mu)}
	=
	2\int_{M} P_t u_0(x) \Delta_x P_t u_0(x)\, \di \mu(x)
	=
	-2\Ecal(P_t u_0),
	\end{equation}
where $\Ecal$ turns out to be the associated Dirichlet form. Note that since $P_{\ker}=E\pa{\br{0}}$ we can rewrite \eqref{eq:diff_ineq} as $\frac{\di}{\di t} \Var(P_t u_0) = -2\Ecal(P_t u_0)$.
Now we seek a pure functional inequality involving $\Vcal$ and   $\Ecal$. In particular (see, for example, \cite[Chapter 3, \S3.2]{Villani2002}), one looks for a functional inequality of the form
	\begin{equation}\label{eq:func-ineq}
	\Ecal(u) \geq \Theta(\Vcal(u)),\qquad\forall u\in D(\Ecal),
	\end{equation}
with an appropriate nonnegative  function $\Theta$. Succeeding in finding such an inequality entails, in view of \eqref{eq:diff_ineq},
	\[\frac{\di}{\di t }\Vcal(P_t u_0)
	\leq
	-2\Theta(\Vcal(P_t u_0))
	\]
from which an explicit rate  is derived.

Returning to the heat semigroup, we notice that the classical Poincar\'e inequality \eqref{eq:poincare_original} is \emph{exactly} a functional inequality of the form of \eqref{eq:func-ineq}. Moreover, the linear connection between the variance and the Dirichlet form yields an exponential rate of decay for $\Var(P_t u_0)$.

\subsection{Relationship to Markov semigroups}
In view of Subsection \ref{sec:classical-poincare},  there is a natural   extension  of the notion of a Poincar\'e inequality to general Markov semigroups.
Let  $\br{P_t}_{t\geq 0}$  be a Markov semigroup on $\Hcal=L^2(M,\di \mu)$ with a generator $-H$, where $H$ is a self-adjoint, non-negative operator, and $\di \mu$ its invariant measure. Then the Poincar\'e inequality, as already defined (Definition \ref{def:poincare}), is
\begin{equation*}\label{eq:poincare_general_H}
\Var(u) \leq C \Ecal(u),\qquad \forall u\in D(\Ecal).
\end{equation*}

The following  well known theorem (see \cite[Theorem 4.2.5]{Bakry2014}) serves as a motivation for our current investigation:
\begin{theorem}\label{thm:convergence_and_poincare}
The following conditions are equivalent:
\begin{enumerate}
\item \label{item:H_poincare}
$H$ satisfies a Poincar\'e inequality with constant $C$.
\item\label{item:spectral_gap}
The spectrum of $H$ is contained in $\br{0}\cup \left[ \frac{1}{C},\infty \right)$.
\item\label{item:rate_of_convergence}
For every  $u\in L^2\pa{M,\di \mu}$ and every $t\geq0$, 
\[\Var(P_tu)\leq e^{-2t/C}\Var(u).\]
\end{enumerate}
\end{theorem}

\subsection{The weak Poincar\'e inequality (WPI)}\label{sec:weak}
It is natural to ask whether one can obtain a generalization of Theorem \ref{thm:convergence_and_poincare}  to generators which lack a spectral gap. We note that a differential operator acting on functions defined  in an unbounded domain (generically) lacks a spectral gap. Our Theorems \ref{thm:main} and \ref{thm:main2} provide an answer to this question, where the Poincar\'e inequality is replaced by some form of a weak Poincar\'e inequality. In the following we provide a brief review of the existing literature on variants of the weak Poincar\'e inequality.

This topic has a very rich history, in particular in the second half of the 20th century. As was hinted in Corollary \ref{cor:nash}, a closely related example is Nash's celebrated inequality \cite{Nash1958}:
	\begin{equation*}\label{eq:nash}
	\|u\|_{L^2}^2
	\leq
	C
	\left(\|\nabla u\|_{L^2}^2\right)^{\frac{d}{d+2}}
	\left(\|u\|_{L^1}^2\right)^{\frac{2}{d+2}},\qquad\forall u\in L^1(\R^d)\cap H^1(\R^d)
	\end{equation*}
where $C>0$ does not depend on $u$. Estimates of the same spirit are then developed in \cite{Carlen1987} for example.

The form of the weak Poincar\'e inequality which we consider (Definition \ref{def:wpi}) first appeared in \cite[Equation (2.3)]{Liggett1991}, where  it is also shown how such a differential inequality leads to an algebraic decay rate. These ideas were then further developed in \cite{Bertini1999,Rockner2001,Wang2002,Wang2004a,Aida2004,Wang2004,Wang2009}. We also refer to \cite{Aida1998} where the notion of a ``weak spectral gap'' is introduced.

In fact, in the influential work of R\"ockner and  Wang \cite{Rockner2001} several  variants of the WPI were introduced. The most general one is
	\begin{equation*}
	\Var(u)
	\leq
	\alpha(r)\Ecal(u)+r\Phi(u),\qquad\forall u\in D(\Ecal),\,r>0,
	\end{equation*}
where $\alpha:(0,\infty)\to(0,\infty)$ is decreasing and $\Phi:L^2(\di\mu)\to[0,\infty]$ satisfies $\Phi(cu)=c^2\Phi(u)$ for any $c\in\R$ and $u\in L^2(\di\mu)$. This is equivalent to our $(\Phi,p)$-WPI whenever $\alpha(r)=Cr^{1-p}$.

{Continuing upon the work of R\"ockner and Wang and their notion of WPI, works on connections between these inequalities and isoperimetry or concentration properties of the underlying measures have been extremely prolific in the probability community. We refer the interested reader to \cite{Cancrini2000, Barthe2005, Bobkov2007, Cattiaux2010, Mourrat2011, Hu2019a}.}
For a recent account of the notions discussed here, and in particular the relationship between functional inequalities and Markov semigroups,  we refer to the book  \cite{Bakry2014}.

\section{Proofs of the theorems}\label{sec:proof}
We first prove the more general Theorem \ref{thm:main2}, and  show how Theorem \ref{thm:main} is a straightforward  corollary.   We then  show how to obtain the decay rates  in Theorem \ref{thm:main1}{, and we conclude with the proof of Proposition \ref{prop:main2_improved}}. For brevity, we  omit  the subscripts  from the functions $\psi_{\Xcal,\Ycal}$, $\Psi_{\Xcal,\Ycal}$ and $g_{\Xcal,\Ycal}$.
\subsection{Proof of Theorem \ref{thm:main2}a}
First we show that an estimate on the density of states near $0$ leads to the WPI \eqref{eq:wpi-implicit}. Let $r_0\in(0,r)$ to be chosen later. Let $\{E(\lambda)\}_{\lambda\geq0}$ be the resolution of the identity of $H$. Let $u\in D(\Ecal)\cap\Xcal\cap\Ycal$. Then:
	\begin{align*}
	\Ecal(u)
	&=
	\int_MuHu\,\di\mu
	=
	\int_Mu\int_{[0,\infty)}\lambda\,\di E(\lambda)u\,\di\mu\notag\\
	&\geq
	\int_Mu\int_{[r_0,\infty)}\lambda\,\di E(\lambda)u\,\di\mu
	\geq
	r_0\int_Mu\int_{[r_0,\infty)}\,\di E(\lambda)u\,\di\mu\notag\\
	&=
	r_0\int_Mu\int_{[0,\infty)}\,\di E(\lambda)u\,\di\mu-r_0\norm{E(\{0\})u}^2_\Hcal-r_0\int_Mu\int_{(0,r_0)}\,\di E(\lambda)u\,\di\mu\notag\\
	&=
	r_0\Var(u)-r_0\int_Mu\int_{(0,r_0)}\,\di E(\lambda)u\,\di\mu.\label{eq:2.1}
	\end{align*}
We now use the estimate on the density of states \eqref{eq:dens-states2} to obtain
	\begin{align*}
	\int_M&u\int_{(0,r_0)}\,\di E(\lambda)u\,\di\mu
	=
	\int_{(0,r_0)}\frac{\di}{\di\lambda}\left(E(\lambda)u,u\right)_\Hcal\di\lambda\\\
	&\leq
	\|u\|_\Xcal\|u\|_\Ycal\int_{(0,r_0)}\psi(\lambda)\,\di\lambda
	=
	\|u\|_\Xcal\|u\|_\Ycal\Psi(r_0).
	\end{align*}
Hence we have
	\begin{equation*}
	\Ecal(u)
	\geq
	r_0\left(\Var(u)-\|u\|_\Xcal\|u\|_\Ycal\Psi(r_0)\right).
	\end{equation*}
Let $K\in(0,1)$ and define
	\begin{equation*}
	r_0=\Psi^{-1}\left(K\frac{\Var(u)}{\|u\|_\Xcal\|u\|_\Ycal}\right)\qquad\text{so that}\qquad \Psi(r_0)=K\frac{\Var(u)}{\|u\|_\Xcal\|u\|_\Ycal}
	\end{equation*}
(to satisfy the condition $r_0<r$ we may need $K$ to be small). Then  we get
	\begin{equation*}
	\Ecal(u)
	\geq
	r_0(1-K)\Var(u)
	\end{equation*}
which completes the  proof.

\subsection{Proof of Theorem \ref{thm:main2} b \& c (and Theorem \ref{thm:main})}
The proofs   follow from the following lemma  where we show how \eqref{eq:wpi-implicit} leads to a $(\Phi,p)$-WPI.

\begin{lemma}\label{lem:thm2-to-thm1}
When $\Xcal=\Ycal$ and $\psi(\lambda)=C_1\lambda^\alpha$, $\alpha>-1$, the inequality \eqref{eq:wpi-implicit} reduces to the  $(\Phi,p)$-WPI  with $\Phi(u)=\|u\|_\Xcal^2$ and $p=\frac{\alpha+2}{\alpha+1}$.  Furthermore, if $\Xcal=\Ycal\subset\Hcal$ we recover Theorem \ref{thm:main}.
\end{lemma}

\begin{proof}
Let $\psi(\lambda)=C_1\lambda^\alpha$, $\alpha>-1$. Then
	\[
	\Psi(\rho)=C_1\int_0^\rho\lambda^{\alpha}\,\di\lambda=\frac{C_1}{\alpha+1}\rho^{\alpha+1}
	\]
so that
	\[
	\Psi^{-1}(\tau)
	=
	\left(\frac{\alpha+1}{C_1}\right)^{\frac{1}{\alpha+1}}\tau^{\frac{1}{\alpha+1}}.
	\]
Hence
	\[
	\Psi^{-1}\left(K\frac{\Var(u)}{\|u\|_\Xcal^2}\right)
	=
	\left(\frac{\alpha+1}{C_1}\right)^{\frac{1}{\alpha+1}}\left(K\frac{\Var(u)}{\|u\|_\Xcal^2}\right)^{\frac{1}{\alpha+1}}.
	\]
Plugging this into \eqref{eq:wpi-implicit} we have
	\begin{align*}
	\Ecal(u)
	&\geq
	(1-K)\Psi^{-1}\left(K\frac{\Var(u)}{\|u\|_\Xcal^2}\right)\Var(u)\\
	&=
	(1-K)\left(\frac{\alpha+1}{C_1}\right)^{\frac{1}{\alpha+1}}\left(K\frac{\Var(u)}{\|u\|_\Xcal^2}\right)^{\frac{1}{\alpha+1}}\Var(u)\\
	&=
	C'\Var(u)^{\frac{\alpha+2}{\alpha+1}}\left(\|u\|_\Xcal^2\right)^{-\frac{1}{\alpha+1}}.
	\end{align*}
This leads to
	\[
	\Var(u)
	\leq
	C''\Ecal(u)^{\frac{\alpha+1}{\alpha+2}}\left(\|u\|_\Xcal^2\right)^{\frac{1}{\alpha+2}}
	\]
which is a $(\Phi,p)$-WPI with $\Phi(u)=\|u\|_\Xcal^2$ and $p=\frac{\alpha+2}{\alpha+1}$.
\end{proof}

\subsection{Proof of Theorem \ref{thm:main1}}
We show that the growth rate assumption \eqref{eq:semigroup-decay-ass}  leads to a  decay of the variance as in \eqref{eq:bound-variance}. {This proof is rather standard and is included for completeness.} Using \eqref{eq:diff_ineq}, the $(\Phi,p)$-WPI and  \eqref{eq:semigroup-decay-ass}, we have:
	\begin{align*}
	\frac{\di}{\di t}\Var(P_tu)
	=
	-2\Ecal(P_tu)
	&\leq
	-2C'\Var(P_tu)^{\frac{\alpha+2}{\alpha+1}}\left(\|P_tu\|_\Xcal^2\right)^{-\frac{1}{\alpha+1}}\\
	&\leq
	-2C'\Var(P_tu)^{\frac{\alpha+2}{\alpha+1}}\left( \|u\|_\Xcal^2+C_2t^\beta\right)^{-\frac{1}{\alpha+1}}
	\end{align*}
where $C'$ is as in the proof of Lemma \ref{lem:thm2-to-thm1}. This is an ordinary differential inequality  for ${y(t):=}\Var(P_tu)$ of the form
	\begin{equation*}
	\dot{y} \leq -A y^{1+a}({B+Ct^{b}})^{-c},
	\end{equation*}
for $a,c,A,B>0$, $b\in\R$, and $C\geq 0$. We readily obtain
	\begin{equation*}\label{eq:diff_inequality_proof_I}
	y(t) \leq {\left(y(0)^{-a} + a A\int_{0}^t {\left(B+Cs^b\right)}^{-c}\,\di s\right)}^{-{1}/{a}}
	\end{equation*}
which yields the  bound  \eqref{eq:bound-variance}. Asymptotically,  we have
	\begin{equation*}
	y(t)= O(t^{-1/a})\text{ as }t\to+\infty,\qquad\text{if } C=0\text{ or }b\leq 0.
	\end{equation*}
Otherwise, it is easy to see that $bc=1$ leads to logarithmic decay, while $bc<1$ leads to polynomial decay. The precise  rates are
	\begin{equation*}
	y(t)=
	\begin{cases}
	O({(\log t) ^{-1/a}})\text{ as }t\to+\infty, & bc=1.\\
	O({t^{-(1-bc)/a}})\text{ as }t\to+\infty, & bc<1.
	\end{cases}
	\end{equation*}
This completes the proof of Theorem \ref{thm:main1}.

\begin{remark}[The constant $C_3$]
It is beneficial to provide a detailed computation of the constant $C_3$ appearing in \eqref{eq:bound-variance}. The following computations are performed up to  a constant $C$ which does not depend on $\alpha, M, \Hcal, \Xcal$ or any other fundamental quantity.

Considering the proof of  Theorem \ref{thm:main1}, we see that $C_3$ is denoted $aA$ where $a=\frac{1}{\alpha+1}$ and $A=2C'$ with $C'=\underbrace{(1-K)K^{\frac{1}{1+\alpha}}}_{\tilde K}\left(\frac{\alpha+1}{C_1}\right)^{\frac{1}{\alpha+1}}$ where $C_1$ and $\alpha$ appear in the bound \eqref{eq:dens-states}. We readily obtain
	\begin{equation*}\label{eq:c3}
	C_3=2\tilde K\frac{1}{\alpha+1}\left(\frac{\alpha+1}{C_1}\right)^{\frac{1}{\alpha+1}}=2\tilde K\left(\alpha+1\right)^{\frac{-\alpha}{\alpha+1}}C_1^{\frac{-1}{\alpha+1}}.
	\end{equation*}
	
{In fact, a short computation using the result of  Proposition \ref{prop:main2_improved} yields the even more explicit formula
\[C_3=2(\alpha+2)^{-\frac{\alpha+2}{\alpha+1}}(\alpha+1)^{\frac{1}{\alpha+1}}C_1^{\frac{-1}{\alpha+1}}.\]}\end{remark}

{\subsection{Proof of Proposition \ref{prop:main2_improved}}
As seen in the Proof of Theorem \ref{thm:main} we have that for all $r_0\in (0,r)$
\begin{equation}\label{eq:to_be_max}
	\Ecal(u)
	\geq
	r_0\left(\Var(u)-\|u\|_\Xcal\|u\|_\Ycal\Psi(r_0)\right).
	\end{equation}
{Our goal is to maximize the right hand side of this inequality.} As such, for any $a,b>0$, consider the function \[h(\rho)=\rho\pa{a - \Psi(\rho)b}.\]By assumption, we can extend $\psi$ to a continuous function on $(0,R)$, so that $h$  is differentiable and we have 
\[h^\prime(\rho)= a-g(\rho) b.\]
As $g$ increases from $0$ to $+\infty$ we see that the unique critical point, $\rho=g^{-1}\pa{\frac ab}$ is a maximum point {of $h$}. Thus
\begin{align*}\max_{\rho\in (0,R)}h(\rho)&=g^{-1}\pa{\frac ab}\pa{a-\Psi\pa{g^{-1}\pa{\frac ab}}b}\\
&=g^{-1}\pa{\frac ab}\pa{a-\rpa{g\pa{g^{-1}\pa{\frac ab}}-g^{-1}\pa{\frac ab}\psi\pa{g^{-1}\pa{\frac ab}}}b}\\
&=g^{-1}\pa{\frac ab}^2\psi\pa{g^{-1}\pa{\frac ab}}b.
\end{align*}}

{Applying this maximization process to the right hand side of \eqref{eq:to_be_max} with $a=\Var(u)$ and $b=\|u\|_\Xcal\|u\|_\Ycal$ yields the desired inequality \eqref{eq:wpi-improved}.}

{To show the second part of the theorem we notice that $\psi(\lambda)$ can be extended to a continuous function on $(0,+\infty)$ with the same formula $C_1\lambda^\alpha$. The expression for $\Psi$ is $\Psi(\rho)=C_1\frac{\rho^{1+\alpha}}{1+\alpha}.$
We note that \[g(\rho)=C_1\frac{2+\alpha}{1+\alpha}\rho^{1+\alpha}\] satisfies the conditions $\lim_{\rho\to0}g(\rho)=0$ and $\lim_{\rho\to+\infty}g(\rho)=+\infty$. Since 
\[g^{-1}(y)=\pa{\frac{1+\alpha}{C_1\pa{2+\alpha}}}^{\frac{1}{1+\alpha}}y^{\frac{1}{1+\alpha}}\]
and
\[g^{-1}(y)^2 \psi\pa{g^{-1}(y)}=C_1\pa{g^{-1}(y)}^{\alpha+2}\]
we obtain the result by substituting  $y=\frac{\Var(u)}{\|u\|_\Xcal^2}$ thus leading to the inequality
\[\Ecal(u) \geq C_1^{-\frac{1}{1+\alpha}}\pa{\frac{1+\alpha}{2+\alpha}}^{\frac{2+\alpha}{1+\alpha}}\pa{\frac{\Var(u)}{\|u\|_\Xcal^2}}^{\frac{2+\alpha}{1+\alpha}}\norm{u}^2_{\Xcal}.\]}

\section{Examples}\label{sec:examples}
Here we consider several notable examples of equations	
	\begin{equation*}
	\left\{\begin{split}
	&\p_tu(t,x)=-H u(t,x), &t\in\R_+,\,x\in\R^d,\\
	&u(0,x)=u_0(x),&x\in\R^d.
	\end{split}\right.
	\end{equation*}
where $H$ is a \textbf{constant coefficient pseudodifferential operator}:
	\begin{equation*}
	H=P(D).
	\end{equation*}
With a slight abuse of notation, we write $H=P(\xi)$, where $\xi\in\R^d$.

\begin{assumption}\label{ass:pdo}
Assume that there exist  ${\gamma_1}>-1$ and $C,{\gamma_2}>0$  so that $P(\xi)$ satisfies the following conditions:
	\begin{enumerate}
	\item
	$P(0)=0$,
	\item
	$C^{-1}|\xi|^{{\gamma_1}+1}\leq P(\xi)\leq C|\xi|^{{\gamma_2}}$, for any $\xi\in\R^d$,
	\item
	$C^{-1}|\xi|^{\gamma_1}\leq|\nabla P(\xi)|$, for any $\xi\in\R^d\setminus\{0\}$,
	\item
	${\Hcal}^{d-1}\left(\{\xi\in\R^d:P(\xi)=\lambda\}\right)\leq C\lambda^{\frac{d-1}{{\gamma_1}+1}}$, for any $\lambda>0$.
	\end{enumerate}
Here ${\Hcal}^{d-1}$ is the $d-1$-dimensional Hausdorff measure. {(We use the same constant $C$ in all inequalities for simplicity, but one could specify different constants)}
\end{assumption}
Then since $P(\xi)$ is  a multiplication operator, one obtains the following simple expression for the spectral measure $E(\lambda)$ of $H$:
	\begin{equation}\label{eq:pdo-bilin}
	\left(E(\lambda) u,v\right)_{L^2}
	=
	\int_{P(\xi)\leq\lambda}\widehat{u}(\xi)\overline{\widehat{v}(\xi)}\,\di\xi.
	\end{equation}
Let $\di\sigma_{\lambda_0}$ denote the uniform Lebesgue measure on the surface $\left\{\xi\in\R^d\,:\,P(\xi)=\lambda_0\right\}$. Then differentiating \eqref{eq:pdo-bilin} and using the coarea formula we obtain:
	\begin{align}
	\left|\frac{\di}{\di\lambda}\Big|_{\lambda=\lambda_0}\left(E(\lambda) u,v\right)_{L^2(\R^d)}\right|\notag
	&=
	\left|\int_{P(\xi)={\lambda_0}}\frac{1}{|\nabla P(\xi)|}\widehat{u}(\xi)\overline{\widehat{v}(\xi)}\,\di\sigma_{{\lambda_0}}\right|\notag\\
	&\leq
	\left|\int_{P(\xi)={\lambda_0}}\frac{1}{|\nabla P(\xi)|}\,\di\sigma_{{\lambda_0}}\right|\|\widehat u\|_{L^\infty(\R^d)}\|\widehat v\|_{L^\infty(\R^d)}\notag\\
	&\leq
	\left|\int_{P(\xi)={\lambda_0}}\frac{1}{|\nabla P(\xi)|}\,\di\sigma_{{\lambda_0}}\right|\|u\|_{L^1(\R^d)}\|v\|_{L^1(\R^d)}\notag\\
	&\leq
	C\int_{P(\xi)={\lambda_0}}|\xi|^{-{\gamma_1}}\,\di\sigma_{{\lambda_0}}\|u\|_{L^1(\R^d)}\|v\|_{L^1(\R^d)}\notag\\
	&\leq
	C\int_{P(\xi)={\lambda_0}}\lambda_0^{-{\gamma_1}/{\gamma_2}}\,\di\sigma_{{\lambda_0}}\|u\|_{L^1(\R^d)}\|v\|_{L^1(\R^d)}\notag\\
	&\leq
	C\lambda_0^{-\frac{{\gamma_1}}{{\gamma_2}}+\frac{d-1}{{\gamma_1}+1}}\|u\|_{L^1(\R^d)}\|v\|_{L^1(\R^d)}.\label{eq:bound-pseudodiff}
	\end{align}
From Theorem \ref{thm:main2}b this bound on the DoS leads to a $(\Phi,p)$-WPI with {$\Xcal=L^1(\R^d)\cap L^2(\R^d)$,} $\Phi(u)=\|u\|_{L^1}^2$ and $p=\frac{2-\frac{\gamma_1}{\gamma_2}+\frac{d-1}{{\gamma_1}+1}}{1-\frac{\gamma_1}{\gamma_2}+\frac{d-1}{{\gamma_1}+1}}=1+\frac{{\gamma_2}({\gamma_1}+1)}{{\gamma_2}{\gamma_1}-{\gamma_1}({\gamma_1}+1)+d{\gamma_2}}$:
	\begin{equation}\label{eq:wpi-pdo}
	\|u\|_{L^2}^2\leq C\Ecal(u)^{\frac{{\gamma_2}{\gamma_1}-{\gamma_1}({\gamma_1}+1)+d{\gamma_2}}{2{\gamma_2}({\gamma_1}+\frac12)-{\gamma_1}({\gamma_1}+1)+d{\gamma_2}}}\|u\|_{L^1}^{2-\frac{2{\gamma_2}{\gamma_1}-2{\gamma_1}({\gamma_1}+1)+2d{\gamma_2}}{2{\gamma_2}({\gamma_1}+\frac12)-{\gamma_1}({\gamma_1}+1)+d{\gamma_2}}}.
	\end{equation}
Moreover, if ${\gamma_2}=1+{\gamma_1}$  the power of $\lambda_0$ in the bound \eqref{eq:bound-pseudodiff} simply becomes $\lambda_0^{\frac{d}{{\gamma_2}}-1}$ and then \eqref{eq:wpi-pdo} simplifies to
	\begin{equation}\label{eq:wpi-pdo2}
	\|u\|_{L^2}^2\leq C\Ecal(u)^{\frac{d}{{\gamma_2}+d}}\|u\|_{L^1}^{\frac{2{\gamma_2}}{{\gamma_2}+d}}.
	\end{equation}

\begin{remark}[Other functional subspaces]
We focus here on solutions lying in $L^1$. However, other natural subspaces to consider are the Hilbert subspaces $L^{2,s}(\R^d)$, defined as
	\begin{equation*}
	L^{2,s}(\R^d):=\left\{u\in L^2(\R^d)\;:\;\|u\|_{L^{2,s}(\R^d)}^2:=\int_{\R^d}|u(x)|^2(1+|x|^2)^{s}\,\di x<\infty\right\}.
	\end{equation*}
These are naturally obtained as follows. In the estimate \eqref{eq:bound-pseudodiff} above, rather than extract $\widehat{u}$ and $\widehat{v}$ in $L^\infty$, one can use the trace lemma to estimate them in $H^s$ with $s>1/2$ {(if the surface is sufficiently regular for the trace lemma to hold)}. Then, one uses the simple observation that the $L^{2,s}$ norm of a function is the same as the $H^s$ norm of its Fourier transform. The main difference is that the power of $\lambda_0$  in the resulting inequality will be different.
\end{remark}

\subsection{The Laplacian.}
For the Laplacian $P(\xi)=|\xi|^2$, the associated equation is the heat equation:
	\begin{equation*}
	\left\{\begin{split}
	&\p_tu(t,x)=\Delta_x u(t,x), &t\in\R_+,\,x\in\R^d,\\
	&u(0,x)=u_0(x),&x\in\R^d.
	\end{split}\right.
	\end{equation*}
Assumption A\ref{ass:pdo} is satisfied with ${\gamma_2}=\gamma_1+1=2$, so that the DoS is estimated by $\lambda_0^{\frac{d}{2}-1}$:
	\begin{equation*}
	\left|\frac{\di}{\di\lambda}\Big|_{\lambda=\lambda_0}\left(E(\lambda) u,v\right)_{L^2}\right|\notag
	\leq
	C\lambda_0^{\frac{d}{2}-1}\|u\|_{L^1(\R^d)}\|v\|_{L^1(\R^d)}.
	\end{equation*}
 Then the WPI \eqref{eq:wpi-pdo2} becomes
	\begin{equation}\label{eq:wpi-pdo3}
	\|u\|_{L^2}^2\leq C\|\nabla u\|_{L^2}^{\frac{2d}{2+d}}\|u\|_{L^1}^{\frac{4}{2+d}}.
	\end{equation}

\begin{remark}[Nash inequality]\label{rek:nash}
This functional inequality is precisely the Nash inequality. This demonstrates how our methodology gives a general framework for many known important inequalities, presented in general form in \eqref{eq:wpi-pdo} and \eqref{eq:wpi-pdo2}.
\end{remark}

{\begin{remark}[The constant in the Nash inequality]
We note that the computation \eqref{eq:bound-pseudodiff} can be performed with precise constants in this case. Then, using Proposition \ref{prop:main2_improved}, we may extract a precise constant in \eqref{eq:wpi-pdo3}. A simple computation yields the constant $C=\pa{\frac{\abs{\Sbb^{d-1}}}{2}}^{\frac{2}{2+d}}\frac{2+d}{d}$. These computations are left to the reader. We note that the optimal constant in the Nash inequality has already been obtained long ago by Carlen and Loss \cite{Carlen1993a}. Improving our constant is the subject of ongoing research.
\end{remark}
}

\textbf{Convergence to equilibrium.}  We can apply Theorems \ref{thm:main2}c and \ref{thm:main1}  with $\alpha=\frac d2-1$ and $\Phi(u)=\|u\|_{L^1}^2$.  Using the fact that the $L^1$ norm of solutions to the heat equation {do not increase}  we have $C_2=0$, where $C_2$ is the constant appearing in \eqref{eq:semigroup-decay-ass}. The bound \eqref{eq:bound-variance} becomes
	\begin{align*}
	\Var(u(t,\cdot))
	&\leq
	\left(\Var(u_0) ^{\frac{-1}{1+\alpha}}+C\int_{0}^t {\norm{u_0}_{L^1}^{\frac{-2}{1+	\alpha}}{\,\di s}}\right)^{-({1+\alpha})}\\
	&=
	\left(\Var(u_0)^{-\frac{2}{d}}+C{\norm{u_0}_{L^1}^{-\frac{4}{d}}}t\right)^{-\frac d2}\\
	&\leq
	C\norm{u_0}_{L^1}^{2}t^{-\frac d2}
	\end{align*}
and we conclude that for every $u_0\in L^1(\R^d)\cap L^2(\R^d)$
	\begin{equation*}
	\|u(t,\cdot)\|_{L^2}^2=\Var(u(t,\cdot))= O(t^{-\frac d2}),\qquad\text{as }t\to+\infty,
	\end{equation*}
which is the optimal rate for the heat equation. This can be extended to any $u_0\in  L^1(\R^d)$ by density.

\subsection{The fractional Laplacian.}
For $P(\xi)=|\xi|^{2p}$ ($p\in(0,1)$) Assumption A\ref{ass:pdo} on $P(\xi)$ is satisfied with  ${\gamma_2}=\gamma_1+1=2p$, so that  the DoS is estimated by $\lambda_0^{\frac{d}{2p}-1}$and \eqref{eq:wpi-pdo2} becomes
	\begin{equation*}\label{eq:wpi-pdo4}
	\|u\|_{L^2}^2\leq C\|(-\Delta)^{\frac p2} u\|_{L^2}^{\frac{2d}{2p+d}}\|u\|_{L^1}^{\frac{4p}{2p+d}}.
	\end{equation*}
	

\begin{remark}
There is no reason not to take values of $p$ greater than $1$. However, the restriction to $p\in(0,1)$ is quite common in the literature, and the result below on time decay only applies to $p\in(0,1)$.
\end{remark}

\textbf{Convergence to equilibrium.}  From \cite{Bonforte2017} we know that $\|u(t,\cdot)\|_{L^1}\leq\|u_0\|_{L^1}$ and as such, much like the previous example, we conclude that
	\begin{equation*}
	\Var(u(t,\cdot))
	\leq
	\left(
	\Var(u_0)^{-\frac{2p}{d}}
	+
	C\norm{u_0}^{-\frac{4p}{d}}_{L^1}t
	\right)^{-\frac{d}{2p}}
	\leq
	C\norm{u_0}_{L^1}^{2}t^{-\frac {d}{2p}}
	\end{equation*}

and hence the asymptotic decay rate is 
	\begin{equation*}
	\|u(t,\cdot)\|_{L^2}^2=\Var(u(t,\cdot))= O(t^{-\frac {d}{2p}}),\qquad\text{as }t\to+\infty.
	\end{equation*}

\subsection{Homogeneous elliptic operators}
{Consider homogenous elliptic operators of the form
	\begin{equation*}
	P(\xi)=\sum_{|\alpha|=m}a_\alpha\xi^\alpha,\qquad m\in\{2,4,6,\dots\},
	\end{equation*}
where $\alpha\in\N^d_0$ is a multi-index with $|\alpha|=\sum_{i=1}^d\alpha_i$ and where all coefficients $a_\alpha\in\R$ are \emph{assumed} to be such that the operator satisfies Assumption A\ref{ass:pdo}. In this case $m=\gamma_1+1=\gamma_2$ and the WPI \eqref{eq:wpi-pdo2} becomes
	\begin{equation*}
	\|u\|_{L^2}^2\leq C\|{P}^{1/2}(D){u}\|_{L^2}^{\frac{2d}{{m}+d}}\|u\|_{L^1}^{\frac{2{m}}{{m}+d}}.
	\end{equation*}}

{Examples of such operators which are not functions of the Laplacian include:
\begin{enumerate}
\item $P(\xi)=\sum_{i=1}^d|\xi_i|^4$,
\item $P(\xi)=\sum_{i=1}^d|\xi_i|^2-\xi_1\xi_2$.
\end{enumerate}
For these examples, the only nontrivial condition to verify is the condition ${\Hcal}^{d-1}\left(\{\xi\in\R^d:P(\xi)=\lambda\}\right)\leq C\lambda^{\frac{d-1}{m}}$.
}
	
{\textbf{Convergence to equilibrium.} In order to prove convergence to an equilibrium state, one has to know how the $L^1$ norm behaves under the flow. The authors are not aware of results in the literature for general operators as the ones we consider here. Based on the known results for the Laplacian and the fractional Laplacian one could ask:
\begin{question}
Is it true that for every homogeneous elliptic operator of order $m$ which satisfies Assumption A\ref{ass:pdo} and which is the generator of a semigroup $(P_t)_{t\geq0}$ there exist $C_2=C_2(u)\geq0$ and $\beta\in\R$ such that for every $t\geq0$, $\norm{P_tu}^2_{L^1}\leq \norm{u}^2_{L^1}+C_2t^\beta$?
\end{question}
If the answer is `yes', from Theorem \ref{thm:main1} this conjecture leads to the following rate of convergence to equilibrium:
	\[
	\Var({P_t u})
	\leq
	\begin{cases} 
	O({(\log t) ^{-\frac dm}}) & \beta=\frac dm.\\
	O({t^{\beta-\frac dm}}) & 0<\beta <\frac dm.\\
	O({t^{-\frac dm}}) & C_2=0 \text{ or }\beta\leq 0.
	\end{cases}
	\]
}




\bibliography{library}

\begin{thebibliography}{10}

\bibitem{Aida1998}
S.~Aida.
\newblock {Uniform Positivity Improving Property, Sobolev Inequalities, and
  Spectral Gaps}.
\newblock {\em Journal of Functional Analysis}, 158:152--185, 1998.

\bibitem{Aida2004}
S.~Aida.
\newblock {Weak Poincar{\'{e}} inequalities on domains defined by Brownian
  rough paths}.
\newblock {\em Annals of Probability}, 32(4):3116--3137, 2004.

\bibitem{Bakry2014}
D.~Bakry, I.~Gentil, and M.~Ledoux.
\newblock {\em {Analysis and Geometry of Markov Diffusion Operators}}, volume
  348 of {\em Grundlehren der mathematischen Wissenschaften}.
\newblock Springer International Publishing, Cham, 2014.

\bibitem{Barthe2005}
F.~Barthe, P.~Cattiaux, and C.~Roberto.
\newblock {Concentration for independent random variables with heavy tails}.
\newblock {\em Applied Mathematics Research eXpress}, 2005:39--60, 2005.

\bibitem{Bertini1999}
L.~Bertini and B.~Zegarlinski.
\newblock {Coercive Inequalities for Gibbs Measures}.
\newblock {\em Journal of Functional Analysis}, 162:257--286, 1999.

\bibitem{Bobkov2007}
S.~G. Bobkov.
\newblock {Large deviations and isoperimetry over convex probability measures
  with heavy tails}.
\newblock {\em Electronic Journal of Probability}, 12:1072--1100, 2007.

\bibitem{Bonforte2017}
M.~Bonforte, Y.~Sire, and J.~L. V{\'{a}}zquez.
\newblock {Optimal existence and uniqueness theory for the fractional heat
  equation}.
\newblock {\em Nonlinear Analysis, Theory, Methods and Applications},
  153:142--168, 2017.

\bibitem{Cancrini2000}
N.~Cancrini and F.~Martinelli.
\newblock {On the spectral gap of Kawasaki dynamics under a mixing condition
  revisited}.
\newblock {\em Journal of Mathematical Physics}, 41(3):1391--1423, 2000.

\bibitem{Carlen1987}
E.~A. Carlen, S.~Kusuoka, and D.~W. Stroock.
\newblock {Upper Bounds for symmetric Markov transition functions}.
\newblock {\em Ann. Inst. Henri Poincar'e Probab. Stat.}, 23(S2):245--287,
  1987.

\bibitem{Carlen1993a}
E.~A. Carlen and M.~Loss.
\newblock {Sharp constant in Nash's inequality}.
\newblock {\em International Mathematics Research Notices}, 1993(7):213--215,
  1993.

\bibitem{Cattiaux2010}
P.~Cattiaux, N.~Gozlan, A.~Guillin, and C.~Roberto.
\newblock {Functional inequalities for heavy tailed distributions and
  application to isoperimetry}.
\newblock {\em Electronic Journal of Probability}, 15:346--385, 2010.

\bibitem{Hu2019a}
S.~Hu and X.~Wang.
\newblock {Subexponential decay in kinetic Fokker-Planck equation: Weak
  hypocoercivity}.
\newblock {\em Bernoulli}, 25(1):174--188, feb 2019.

\bibitem{Liggett1991}
T.~M. Liggett.
\newblock {{\$}L{\_}2{\$} Rates of Convergence for Attractive Reversible
  Nearest Particle Systems: The Critical Case}.
\newblock {\em The Annals of Probability}, 19(3):935--959, jul 1991.

\bibitem{Mourrat2011}
J.~C. Mourrat.
\newblock {Variance decay for functionals of the environment viewed by the
  particle}.
\newblock {\em Annales de l'institut Henri Poincare (B) Probability and
  Statistics}, 47(1):294--327, 2011.

\bibitem{Nash1958}
J.~F.~F. Nash.
\newblock {Continuity of Solutions of Parabolic and Elliptic Equations}.
\newblock {\em American Journal of Mathematics}, 80(4):931, oct 1958.

\bibitem{Rockner2001}
M.~R{\"{o}}ckner and F.-Y. Wang.
\newblock {Weak Poincare Inequalities and L2-Convergence Rates of Markov
  Semigroups}.
\newblock {\em Journal of Functional Analysis}, 185:564--603, 2001.

\bibitem{Villani2002}
C.~Villani.
\newblock {A review of mathematical topics in collisional kinetic theory}.
\newblock In S.~Friedlander and D.~Serre, editors, {\em Handbook of
  mathematical fluid dynamics}, volume~1. Elsevier, 2002.

\bibitem{Wang2002}
F.-Y. Wang.
\newblock {Functional Inequalities and Spectrum Estimates : The Infinite
  Measure Case}.
\newblock {\em Journal of Functional Analysis}, 194:288--310, 2002.

\bibitem{Wang2004a}
F.-Y. Wang.
\newblock {Functional inequalities on abstract Hilbert spaces and
  applications}.
\newblock {\em Mathematische Zeitschrift}, 246(1-2):359--371, 2004.

\bibitem{Wang2004}
F.-Y. Wang.
\newblock {Weak Poincar{\'{e}} inequalities on path spaces}.
\newblock {\em International Mathematics Research Notices}, 2004(2):89--108,
  2004.

\bibitem{Wang2009}
F.-y. Wang.
\newblock {Super and weak Poincar{\'{e}} inequalities for hypoelliptic
  operators}.
\newblock {\em Acta Mathematicae Applicatae Sinica, English Series},
  25(4):617--630, 2009.

\end{thebibliography}
\bibliographystyle{abbrv}
\end{document}